\documentclass[a4paper,10pt,reqno]{amsart}
\usepackage{amsthm,amssymb,latexsym,amsmath}

\usepackage[utf8]{inputenc}
\usepackage{amsmath,amssymb,amsfonts,amsthm}
\usepackage{MnSymbol}
\usepackage{multirow}

\newtheorem{theorem}{Theorem}[section]

\newtheorem{proposition}[theorem]{Proposition}

\newtheorem{lemma}[theorem]{Lemma}
\theoremstyle{definition}

\numberwithin{equation}{section}
\usepackage{color}

\newcommand{\F}{{\mathcal{F}}}
\newcommand{\G}{{\mathcal{G}}}

\newcommand{\C}{\mathbb{C}}

\newcommand{\OO}{\mathcal{O}}

\newcommand{\PP}{\mathbb P}

\newcommand{\Z}{\mathbb Z}
\newcommand{\aut}{\operatorname{Aut}}

\newcommand{\gl}{\operatorname{GL}}
\newcommand{\pgl}{\operatorname{PGL}}
\newcommand{\SL}{\operatorname{SL}}
\newcommand{\sing}{\operatorname{Sing}}

\newcommand{\dv}{\operatorname{div}}
\newcommand{\tr}{\operatorname{tr}}

\newcommand{\Pf}{\mathcal{P}}
\newcommand{\Hf}{\mathcal{H}}

\title[Foliations with finite group of symmetries]{Foliations on the projective plane with finite group of symmetries}
\date{\today}

\author[Muniz]{Alan Muniz}

\author[Rosas]{Rudy Rosas}

\address[A. Muniz]{ Programa de Pós-Graduação em Matemática -- CCE, Universidade Federal do Espírito Santo -- UFES, Av. Fernando Ferrari 514, 29075-910, Vitória -- ES, Brasil.}

\curraddr{ Instituto de Matemática e Estatatística
  \\ Universidade Federal Fluminense, UFF -- Rua Prof. Marcos Waldemar de Freitas Reis, S/N -- Bloco H, 24210-201,
  Niterói--RJ, Brasil.}
\email{alannmuniz@gmail.com}

\address[R. Rosas]{ Departamento de Ciencias, Pontificia Universidad Cat\'olica del Per\'u, Av. Universitaria 1801, Lima, Per\'u}\email{rudy.rosas@pucp.pe}

\subjclass[2010]{ 37F75 - 32M25 - 20B25} \keywords{Automorphism -
Holomorphic foliations - Finite Groups - Invariant Theory}

\begin{document}

\begin{abstract}
Let $\mathcal{F}$ denote a singular holomorphic foliation on $\mathbb{P}^2$ having a  finite automorphism group $\aut(\mathcal{F})$.  Fixed the degree of $\mathcal{F}$, we determine the maximal value that $|\aut(\mathcal{F})|$ can take and explicitly exhibit all the foliations attaining this maximal value.  Furthermore, we classify the foliations with large but finite automorphism group. 
\end{abstract}

\maketitle

\section{Introduction}
The presence of symmetries has been used to understand some relevant problems in holomorphic foliation theory, especially concerning integrability. For instance, it is used in the construction of Jouanolou's foliations \cite{J}, which play an important role in his proof of the density of foliations on the complex projective plane $\PP^2$ without invariant algebraic curves; in \cite{PSA} this relation between automorphism groups and integrability was explored and put into more concrete terms. 

A new perspective for the study of these automorphism groups arises after the birational classification of foliated surfaces, given independently by McQuillan and Mendes -- see \cite{Bru}. This is a foliated counterpart of the classical Enriques-Kodaira work on surfaces. From the works of Hurwitz and Klein for curves and the work of Andreotti in higher dimension, the question about the finiteness of automorphism groups of foliations appears naturally. In 2002, Pereira and Sánchez \cite{PS} proved a foliated version of Andreotti's theorem: foliated surfaces of general type have finite self-bimeromorphism groups. Then arises the question: how large these finite groups can be? To propose a more precise question we  
restrict our study to the case of singular holomorphic foliations  on $\PP^2$: is there a function $f\colon\mathbb{N}\to\mathbb{N}$ such that, for any singular holomorphic foliation $\mathcal{F}$ of degree $d$  on $\PP^2$ with $|\aut(\F)|<+\infty$ we have \[ |\aut(\F)|\le f(d) ?\]
if the response is yes, what is the minimal function $f$?
In \cite{autfol}, Corrêa and the first author of this paper show that the function $f(d)=3(d^2+d+1)$ bounds the order of $\aut(\F)$ for a generic class of foliations $\mathcal{F}$ of degree $d$. The maximal value $3(d^2+d+1)$ in this class is attained by the Jouanolou Foliation  \cite{J}.
In the present work, we show that  the bound $3\left(d^2+d+1\right)$ does not work for all foliations but we prove that the function $f$ indeed exists. First of all, recall that foliations of degree $0$ or $1$ have infinitely many automorphism, so we only need to look by a function $f$ defined on $\{2,3,\ldots\}$.
\begin{theorem}\label{f} Let $f\colon \{2,3,\ldots\}\to\mathbb{N}$ be defined by  $f(2)=24$, $f(3)=39$, $f(4)=216$ and $f(d)=6(d-1)^2$ for $d\ge 5$. Then, if $\F$ is a degree $d$ foliation on $\PP^2$ such that $|\aut(\F)| <	+\infty$, we have
\[ |\aut(\F)|\le f(d). \]
Moreover, this bound is sharp: there exist foliations attaining this bound for any degree $d\ge 2$.
\end{theorem}
Furthermore, our main goal is to give a complete classification of the foliations with large but finite automorphism group. In particular, we present all the foliations attaining the bound $f(d)$ in Theorem \ref{f}. 

\begin{theorem}\label{thmp}
	Let $\F$ be a foliation of degree $d\ge 2$ on $\PP^2$  such that 
	\[
	3(d^2+d+1) \le |\aut(\F)| \le f(d).	
	\]
	Then $\F$ is projectively equivalent to one of the foliations in the following table --  $\overline{T}$  and $\overline{I}$ are the binary tetrahedral and icosahedral groups, respectively.\\

\begin{table}[h]
	\centering
	\renewcommand{\arraystretch}{1.2}
	\begin{tabular}{| l | c | c | c | l |}
	
		\hline
		Fol.& degree & $|\aut(\F)|$ & $\aut(\F)$  &  description \\ 
		\hline
		$\mathcal{J}_d$	& $ d$ & $3(d^2+d+1)$ & $\Z/(d^2 +d + 1)\Z\rtimes \Z/3\Z$  & Jouanolou foliation\\
		\hline
		$\G_d$	& \multirow{2}{*}{$d \geq 5$} & \multirow{2}{*}{$6(d-1)^2$} &  \multirow{2}{*}{$(\Z/(d-1)\Z)^2\rtimes {S}_3$}   & nonisotrivial hyperbolic fibration  \\ \cline{1-1}\cline{5-5}
		$\F_d$	&  &  &   & isotrivial hyperbolic fibration \\[0.1ex] 
		\hline
		$\Pf_5$	& $5$ & $96$ & $\left(\Z/2\Z \times \overline{T}\right)\cdot \Z/2\Z$ & \multirow{2}{*}{general type Bernoulli foliation } \\ \cline{1-4}
		$\Pf_{11}$	& $11$ & $600$ & $\Z/5\Z \times\overline{I}$  &  \\ \hline
		$\mathcal{S}$	& $ 2$ & $24$ & $(\Z/2\Z)^2\rtimes {S}_3$  &  rational fibration\\ 
		
		\hline
		$\Hf_4$	& $4$ & \multirow{2}{*}{$216$} & \multirow{2}{*}{Hessian group}   & nonisotrivial elliptic fibration \\ \cline{1-2}\cline{5-5}
		$\Hf_7$	& $7$ & &   & nonisotrivial hyperbolic fibration \\		\hline
	\end{tabular}
	\label{tabfol}
	\end{table}
	In particular, the extremal equality $|\aut(\F)|=f(d)$ is attained by the fibration $\mathcal{S}$ if $d=2$, the Jouanolou foliation $\mathcal{J}_3$ if $d=3$, the foliation $\mathcal{H}_4$ if $d=4$, the foliations $\mathcal{G}_5$,
	 $\mathcal{F}_5$ and $\mathcal{P}_5$ if $d=5$, and the foliations  $\mathcal{G}_d$ and $\mathcal{F}_d$ if $d\ge 6$. 
\end{theorem}

Besides the well known Jouanolou foliations $\mathcal{J}_d$,  the foliations $\F_d$, $\Hf_4$ and $\Hf_{11}$ are already described in the literature: they appear in \cite{MP} as examples of reduced convex foliations; these are, respectively, the Fermat foliations $\F_d$,  the Hesse pencil of cubics $\Hf_4$ and the foliation $\Hf_7$ associated to the extended Hesse configuration of lines.  The other foliations are described in Section \ref{examples} but it is worth mentioning here that, with the only exception of $\mathcal{J}_d$, all foliation in  Table \ref{tabfol} have some kind of first integral: the Bernoulli foliations have Liouvillian first integrals, and all the others have rational first integrals. Moreover, 
from Table \ref{tabfol} we can derive the following fact:
since the canonical bundle of $\F$ is $K_\F = \OO_{\PP^2}(d-1)$, we have a linear bound for $|\aut(\F)|$ in terms of the Chern number $K_\F^2$ -- this kind of bounds are obtained in \cite{CF} for foliations with ample canonical bundle on projective surfaces.

This paper is organized as follows. In Section \ref{definiciones} we recall some basic definitions and results about foliations. In Section \ref{examples} we describe the foliations already mentioned in our main theorems. In Section \ref{class} we reduce the proof of our main theorems to a variety of specific results according to the classification of finite groups of $\operatorname{PGL}(3,\mathbb{C})$. The remaining sections are devoted to prove the results presented in Section \ref{class}. It is worth mentioning that in Section \ref{moliensec} we develop a Molien-type formula -- Theorem \ref{molienthm} -- that is fundamental  to analyze the  transitive primitive groups. This formula leads to some simple but very long calculations that we have managed to do with Maple\texttrademark.

\subsection*{Acknowledgements} \small We would like to thank the referee for the suggestions, comments, and improvements to the exposition. The first named author was financed in part by the Coordena\c c\~ao de Aperfei\c coamento de Pessoal de N\'ivel Superior - Brasil (CAPES) - Finance Code 001. The second named author was supported  by Vicerrectorado de Investigaci\'on de la Pontificia Universidad Cat\'olica del Per\'u.

\section{Preliminaries}\label{definiciones}

\subsection*{Foliations on $\PP^2$}

In this paper we consider  {singular holomorphic foliations by curves} -- foliations for short -- with isolated singularities on $\PP^2$. Such a foliation can be defined in an affine chart $\mathbb{C}^2$ with coordinates $(x,y)$ by a polynomial vector field $P\partial_x+Q\partial_y$ with isolated singularities or even by a polynomial 1-form, for example the form 
$Qdx-Pdy$. Conversely, a polynomial vector field or 1-form with isolated singularities on $\mathbb{C}^2$ define a unique holomorphic foliation with isolated singularities on $\mathbb{P}^2$.  The degree of a foliation $\mathcal{F}$ -- denoted by $\deg(\F)$ -- is geometrically defined as the number of tangencies of $\mathcal{F}$ with a generic line in $\mathbb{P}^2$. This degree  reflects  on the  polynomial 1-form defining $\mathcal{F}$ in the following way. 
Let $\omega$ be a polynomial 1-form with isolated singularities defining $\F$ and write it in the form
$$\omega=\sum\limits_{j=0}^k(A_jdx+B_jdy),$$ where $A_j$ and $B_j$ are homogeneous polynomials of degree $j$, and $A_kdx+B_kdy\neq0$. Then  we have the following possibilities:
\begin{enumerate}
\item The polynomial $A_kx+B_ky$ is not zero; in this case the  line at infinity of $\mathbb{P}^2$ is invariant by $\mathcal{F}$ and we have  $\deg(\mathcal{F})=k$. 
\item The polynomial $A_kx+B_ky$ is zero; in this case the line  at infinity of $\mathbb{P}^2$ is generically transverse to $\mathcal{F}$ and we have  $\deg(\mathcal{F})=k-1$.
\end{enumerate}

In a more global point of view, a foliation $\mathcal{F}$ of degree $d$ can also be defined  by  homogeneous vector field in three variables: there exists a polynomial vector field 
 $$v=A\partial_X+B\partial_Y+C\partial_Z$$ of degree $d$ in $\mathbb{C}^3$ which  induce -- via the natural projection $\mathbb{C}^3\backslash\{0\}\to \mathbb{P}^2$ --  a singular holomorphic distribution of lines on $\mathbb{P}^2$ coinciding with the tangent distribution of $\mathcal{F}$.
 Another homogeneous vector field $w$ of degree $d$ induce  the same foliation $\mathcal{F}$ if and only if there exist
  $\alpha \in \C^{\ast}$ and a polynomial $P$ of degree $d-1$ such that $$w= \alpha v +P(X\partial_X+Y\partial_Y+Z\partial_Z).$$ From this it is easy to see that we can find a homogeneous vector field $w$ defining $\F$ and such that $\dv(w)=0$, where $\dv(w)$ is the divergence of $w$ with respect to $dX \wedge dY \wedge dZ$. This choice of vector field is called the Darboux normal form and it is unique -- up to scalar multiplication, see \cite[p.6]{J} and \cite[p.63]{D}

\subsection*{Automorphisms}

 A biholomorphism $\varphi \colon \mathbb{P}^2\to\mathbb{P}^2$ is an automorphism of $\F$ if $\varphi$ maps each leaf of $\F$ onto a leaf 
 of $\F$. The set of automorphisms of $\F$ is a group, which is denoted by $\aut(\F)$. If $v$ is a homogeneous vector
  field in $\mathbb{C}^3$ defining $\F$, then the image of an element  $\varphi \in \gl(3,\C)$ in $\pgl(3,\C)$  is an automorphism of $\F$ if and only if the pushforward $\varphi_*v=\varphi \cdot (v\circ \varphi^{-1})$ is also a 
  vector field defining $\F$.
 Hence, given a finite subgroup $G$ of $\gl(3,\C)$, its image in $\pgl(3,\C)$ preserves the foliation  defined by $v$ with $\dv(v) =0$  if there
  exists a character $\chi: G \longrightarrow \C^{\ast}$ such that, for all $\varphi\in G$,
\[
\varphi_{\ast}v = \chi(\varphi)v. 
\]
If this is the case, we also say that the vector field $v$ is $G$-semi-invariant. 
\section{Examples}\label{examples}
In this section we describe those foliations already mentioned in Theorem  \ref{thmp}. Hereafter we will work with homogeneous coordinates $(X:Y:Z)$ on $\mathbb{P}^2$ and $(x,y)$ will denote the natural coordinates in $\{Z \neq 0\}$.
\subsection{The Jouanolou Foliation $\mathcal{J}_d$}

The Jouanolou foliation of degree $d$, denoted by $\mathcal{J}_d$, is defined by the vector field
\[
	Y^d\partial_X  + Z^d\partial_Y + X^d\partial_Z
	\]
or, in affine coordinates, by the 1-form 
$$(x^dy-1)dx+(y^d-x^{d+1})dy.$$ The automorphism group of $\mathcal{J}_d$  has order $3(d^2+d+1)$ and it is generated by  the maps
 $[X:Y:Z]\mapsto[Y:Z:X]$ and
$[X:Y:Z]\mapsto [l^{d+1}X:lY:Z]$,
where $l^{d^2+d+1}=1$ with $l$ primitive. 
\subsection{The Fermat Foliation $\F_d$}
The Fermat foliation $\F_d$ of degree $d$  --  already described in \cite{MP} -- is  generated by the vector field 
$$X^d\partial_X +Y^d\partial_Y + Z^d\partial_Z.$$  In affine coordinates we can define $\F_d$ by the 1-form 
$$\left(y-y^d\right)dx-\left(x-x^d\right)dy$$ or even  by the rational first integral 
$$\frac{x^{1-d}-1}{y^{1-d}-1}.$$ This foliation has an automorphism group of order $6(d-1)^2$, generated by the transformations
\begin{align*}
&[X:Y:Z]\mapsto[Y:Z:X],\\  
&[X:Y:Z]\mapsto [X:Z:Y],\\ 
 &[X:Y:Z]\mapsto[\lambda X:Y:Z] \textrm{ and }\\ 
 &[X:Y:Z]\mapsto[X:\lambda Y:Z],
 \end{align*}  where $\lambda$ is a primitive $(d-1)$th root of unity.
\subsection{The Foliation $\mathcal{G}_d$} In affine coordinates, 
this degree $d$ foliation  is defined
by the rational first integral
 \[
 \frac{\left(x^{d-1}+y^{d-1}+1\right)^3}{x^{d-1}y^{d-1}}
\]
or by  the  the 1-form 
$$\left(y+ y^d-2 x^{d-1}y\right)dx
 +\left(x+x^d-2xy^{d-1}\right)dy.$$ 
This foliation has the same group of automorphism as the Fermat foliation $\F_d$. 
For $d\geq 5$, the foliation $\G_d$ is birational to a nonisotrivial hyperbolic fibration and is therefore of general type.

\subsection{The Hessian Foliations $\mathcal{H}_4$ and $\mathcal{H}_7$}\label{hesse}

 The foliation $\Hf_4$ is the degree $4$ foliation given by the well known Hesse pencil of cubics. This foliation has the rational first integral 
\[
[X:Y:Z]\mapsto [X^3+Y^3+Z^3: XYZ].
\]
On the other hand, the foliation $\mathcal{H}_7$ -- see \cite{MP} -- is the degree $7$ foliation given in affine coordinates by the vector field
\begin{align*}
 (x^3-1)(x^3+7y^3+1)x\partial_x+(y^3-1)(y^3+7x^3+1)y\partial_y.
\end{align*}
This foliation leaves invariant the extended Hessian arrangement of $21$ lines on the plane and it is tangent to a pencil of curves of degree $72$.  The automorphism group of both foliations $\mathcal{H}_4$ and $\mathcal{H}_7$ is  the classical Hessian Group -- see \cite{BLI} -- of order $216$, generated by the maps \begin{align*} &[X:Y:Z]\mapsto[Y:Z:X],\\ &[X:Y:Z]\mapsto[X:Y:\lambda Z],\\
&[X:Y:Z]\mapsto[X:\lambda Y: \lambda^2 Z] \textrm{ and }\\ &[X:Y:Z]\mapsto[X+Y+Z:X+\lambda Y+ \lambda^2 Z: X+\lambda^2 Y+\lambda Z],\end{align*}
where $\lambda$ is a primitive cubic root of unity. 
 See  \cite{MP} for more details about the foliations $\mathcal{H}_4$ and $\mathcal{H}_7$. 
\subsection{The Rational Fibration $\mathcal{S}$}
The foliation $\mathcal{S}$ is the foliation of degree $2$ defined by the vector field
$$YZ\partial_X + ZX\partial_Y + XY\partial_Z.$$ In affine coordinates $\mathcal{S}$ is defined by the 1-form 
$$\left(-x+xy^2\right)dx+\left(y-x^2y\right)dy$$ or by the rational first integral 
$$\frac{1-x^2}{1-y^2}.$$ This foliation has the same group -- of order $24$ -- as the Fermat foliation $\F_3$. 
\subsection{The Bernoulli Foliations $\mathcal{P}_5$ and $\mathcal{P}_{11}$}
The degree $5$ foliation $\mathcal{P}_5$ and the degree $11$ foliation $\mathcal{P}_{11}$  are defined  by the 1-forms
\begin{align*}
xdy-ydx+d\big(x^5y-xy^5\big)\end{align*}
and
	\begin{align*}
	&xdy-ydx+d\big(x^{11}y+11x^6y^6-xy^{11}\big),
	\end{align*} 
	respectively. In the affine coordinates $(x,y)$ above,
$\aut(\mathcal{P}_5)$ is  a subgroup of order $96$ of $\operatorname{GL}(2,\mathbb{C})$. The center of this group is  $\{\lambda I\colon \lambda^4=1\}$ and its image  in $\operatorname{PGL}(2,\mathbb{C})$ is   the octahedral group presented in such way its smallest orbit is given by the 6 lines $$x^5y-xy^5=0.$$  
On the other hand, the group $\aut(\mathcal{P}_{11})$ is the subgroup of order $600$ of 
$\operatorname{GL}(2,\mathbb{C})$ having $\{\lambda I\colon \lambda^{10}=1\}$ as its center and whose image in $\operatorname{PGL}(2,\mathbb{C})$ is  the icosahedral group presented in such way  its smallest orbit is given by the 12 lines $$x^{11}y+11x^6y^6-xy^{11}=0.$$ 
The foliations  $\mathcal{P}_5$ and $\mathcal{P}_{11}$ have no rational first integrals but they are particularly special: since both foliations are presented in  the form  \begin{equation}\label{bu}xdy-ydx+dP=0,\end{equation}  
where $P$ is a square free homogeneous polynomial of degree $\ge 4$, they have nondegenerate singularities and both foliations are of general type. Moreover, if we do the change of coordinates $(x,y)=(\frac{t}{z},\frac{1}{z})$ in Equation \eqref{bu} we 
 obtain the Bernoulli equation
\[
nP(t,1)\frac{dz}{dt} = {P_x}(t,1) z-z^{n-1},
\]
so  $\mathcal{P}_5$ and $\mathcal{P}_{11}$ have Liouvillian first integrals -- see \cite{CS}.

\section{Finite subgroups of $\pgl(3,\C)$ and organization of results}\label{class}
In this section, we present the classification of finite subgroups of $\pgl(3,\C)$  given by Blichfeldt in \cite{BLI} -- another great exposition is given by Yu and Yau in \cite{gor}. For each type of group of this classification we state a specific result about automorphisms of foliations;
thus \begin{itemize} \item in all this section $\F$ denote a degree $d$
 holomorphic foliation with isolated singularities on $\mathbb{P}^2$  such that
  $\aut(\F)$ is finite.\end{itemize} It will be easy to see that theorems \ref{f} and 
  \ref{thmp} are  direct consequences of the results presented in this section, so we leave this verification to the reader. 

 In \cite[Chapter XII]{BLI}, the finite subgroups of $\pgl(3,\C)$ are classified as follows. There are two kinds of intransitive groups: abelian groups generated by diagonal matrices and finite non-abelian subgroups of ${\rm GL}(2,\C)$, they are described in subsections \ref{ag} and \ref{gg}. The transitive groups are classified in two classes: the infinite class of imprimitive groups -- which we describe in  Subsection \ref{tig} -- and  the finite class of primitive groups  which we describe in subsections \ref{tpg} and \ref{tpsg}.

\subsection{Abelian groups of diagonal matrices}\label{ag}
A group in this class is such that, in suitable homogeneous coordinates, all its elements are of the form
$$[X:Y:Z]\mapsto [aX:bY:Z],$$
where $a,b\in\mathbb{C}^*$. About the foliations having such a group of automorphisms, we have the following result which is a corollary of Proposition \ref{pro1}.
\begin{proposition}Suppose that $\aut (\mathcal{F})$ is a finite abelian group of diagonal matrices. Then $$|\aut(\F)|\le d^2+d+1$$
\end{proposition}
\subsection{Finite non abelian subgroups of $\operatorname{GL}(2,\mathbb{C})$}\label{gg}
In suitable homogeneous coordinates, a group in this class is composed by transformations of the form 
$$[X:Y:Z]\mapsto \left[aX+bY:cX+dY:Z\right],$$ where $\begin{bmatrix}a&b\\c&d\end{bmatrix}$ 
 belongs to a finite non abelian subgroup $G$ of $\operatorname{GL}(2,\mathbb{C})$.
 Recall that the image of $G$ in $\operatorname{PGL}(2,\mathbb{C})$, being finite, is one of the following groups:
 \begin{enumerate}
	\item A dihedral group.
	\item A polyhedral group, so we have three possibilities: 
	\begin{enumerate}\item The Tetrahedral Group $T$, isomorphic to the alternating group ${A}_4$;
	\item The Octahedral Group $O$, isomorphic to the symmetric group ${S}_4$;
	\item The Icosahedral Group $I$, isomorphic to the alternating group ${A}_5$.
	\end{enumerate}
\end{enumerate}
In relation with the dihedral case, in Section \ref{sec2} we prove the following result.   
\begin{proposition} \label{dg}Suppose  that the image of $\aut (\mathcal{F})$ in $\operatorname{PGL}(2,\mathbb{C})$ is a dihedral group. 
Then $$|\aut(\F)|\le 2(d^2+d+1).$$
\end{proposition}
For the polyhedral case we have the following result, proved in Section \ref{poly}.
\begin{proposition}\label{polyhedral} Suppose  that the image of $\aut (\mathcal{F})$ in $\operatorname{PGL}(2,\mathbb{C})$ is a polyhedral group. Then
 $$|\aut(\F)|\ge 3(d^2+d+1)$$ if and only if $\F$ is equivalent to one of the Bernoulli foliations $\mathcal{P}_5$ or $\mathcal{P}_{11}$.
\end{proposition}
\subsection{Transitive imprimitive groups}\label{tig}
 In this case we have two types of groups. The first type is composed by those groups such that, in suitable homogeneous coordinates, they are generated 
 by the transformation 
 $$ [X:Y:Z]\mapsto [Y:Z:X]$$ and a group of transformations of the
  form $$[X:Y:Z]\mapsto [aX:bY:Z].$$ 
 \begin{proposition} \label{tig1}Suppose  that $\aut (\mathcal{F})$ is a group of the type above.  Then
 $$|\aut(\F)|\le 3(d^2+d+1)$$ and the equality holds if and only if 
$\mathcal{F}$ is isomorphic to the Jouanolou foliation $\mathcal{J}_d$. 
\end{proposition}

The second type of transitive imprimitive groups  is composed by those groups such that,  in suitable homogeneous coordinates, they are generated by 
$$[X:Y:Z]\mapsto [Y:Z:X],$$ a transformation  of the form $$[X:Y:Z]\mapsto [\mu Y:\nu X:Z],$$ where $\mu,\nu\in\mathbb{C}^*$, and a group of transformations of the form $$[X:Y:Z]\mapsto [aX:bY:Z].$$ 

\begin{proposition}\label{nd}  Suppose  that $\aut (\mathcal{F})$ is a group of the type above.

\begin{enumerate}
\item Assume that the line at infinity $Z=0$ is invariant by $\mathcal{F}$. Then 
$$|\aut(\F)|\le 6(d-1)^2$$  and the equality holds if and only if  $\mathcal{F}$ is isomorphic to $\F_d$ or  $\G_d$.
Moreover, if $\mathcal{F}$ is different from $\F_d$ and $\G_d$ we have
$$|\aut(\F)|< 3(d^2+d+1).$$
\item If the line at infinity  is generically transverse to $\mathcal{F}$, then 
  $$|\aut(\F)|\ge 3(d^2+d+1)$$ if and only if $\F$ is equivalent to the foliation $\mathcal{S}$. 
 \end{enumerate}
\end{proposition}
Proposition  \ref{tig1} and Proposition \ref{nd} are proved in Section \ref{stig}.
\subsection{Transitive primitive groups having a non-trivial normal subgroup}\label{tpg}
Up to isomorphism, we have three groups in this class: 
\begin{enumerate}
\item The Hessian Group $G$, already presented in Subsection \ref{hesse}.
\item  The normal subgroup $E\triangleleft G$ generated by the transformations
\begin{align*} &T[X:Y:Z]=[Y:Z:X],\\ 
&S[X:Y:Z]=[X:\lambda Y: \lambda^2 Z],\textrm{ and }\\ &V[X:Y:Z]=[X+Y+Z:X+\lambda Y+ \lambda^2 Z: X+\lambda^2 Y+\lambda Z].\end{align*}
\item The normal subgroup $F\triangleleft G$ generated by the group $E$ and the map $UVU^{-1}$, where
\begin{align*} U[X:Y:Z]=[X:Y:\lambda Z].
\end{align*}
\end{enumerate} 
\begin{proposition} \label{molien1} Suppose  that $\aut (\mathcal{F})$ is a transitive primitive group having a non-trivial normal subgroup.   Then, we have 
 $$|\aut(\F)|\ge 3(d^2+d+1)$$ if and only if $\F$ is equivalent to one of the Hessian foliations $\mathcal{H}_4$ or $\mathcal{H}_7$.  
\end{proposition}
This result is proved in Section \ref{primitive}.

\subsection{Transitive primitive simple groups}\label{tpsg}
Up to isomorphism, we have three groups in this class: The icosahedral group which is isomorphic to $A_5$,  the image of the Valentiner Group in $\pgl(3,\C)$ which is isomorphic to $A_6$ and the Klenian group isomorphic to $\operatorname{PSL}(2,\mathbb{F}_7)$. In this case we have the following result,  proved in Section \ref{primitive}. 

 \begin{proposition}\label{molien2} Suppose  that $\aut (\mathcal{F})$ is a transitive primitive simple group.   Then 
 $$|\aut(\F)|< 3(d^2+d+1).$$  
\end{proposition}

\section{Foliations with finite group of diagonal automorphisms}\label{sec2}
In this section we study foliations on $\mathbb{P}^2$  having finitely many automorphisms of the form 
 $$[X:Y:Z]\mapsto [aX:bY:Z]$$
 which correspond to $(x,y) \mapsto (ax,by)$ in the natural affine coordinates on $\{Z\neq 0\}$.
 
Consider a polynomial 1-form $\omega=Adx+Bdy$  in $\mathbb{C}^2$. We assume that  $A$ and $B$ have no common factor in   $\mathbb{C}[x,y]$, so that $\omega$ has isolated singularities.  
We can write 
$$\omega=\sum\limits_{i,j\ge 0} \alpha_{ij} x^i y^j dx +
\sum\limits_{i,j\ge 0}  \beta_{ij} x^i y^j dy,$$ where $\alpha_{ij},\beta_{ij}\in\mathbb{C}$. To each monomial  form $ \alpha_{ij} x^i y^j dx$, $\alpha_{ij}\neq 0$ or  $\beta_{ij} x^i y^j dy$, $\beta_{ij}\neq 0$ appearing in $\omega$, associate respectively the monomial $x^{i+1} y^j$ or $x^{i} y^{j+1}$ and consider the set 
$$\mathcal{M}=\{x^{i+1} y^j\colon \alpha_{ij}\neq 0\}\cup 
\{x^{i} y^{j+1}\colon \beta_{ij}\neq 0\}.$$ Set $$\mathcal{M}'=\{\mathfrak{m}-\mathfrak{m}'\colon \mathfrak{m}, \mathfrak{m}'\in \mathcal{M}\}.$$ Given $f\in\mathcal{M}'$ nonzero, write $f=x^ky^l\tilde{f}$ with $k,l\ge 0$ and $(\tilde{f},xy)=1$, and  define the set $$S=\{\tilde{f}\colon f\in\mathcal{M}'\}.$$  

 Let $\mathcal{F}$ be the singular holomorphic foliation defined by $\omega$ on $\mathbb{C}^2$ and let $G$ be the group of transformations  of the form $(x,y)\mapsto (ax,by)$, $a,b\in\mathbb{C}^*$ that leave $\mathcal{F}$ invariant. 
 Precisely,  $G$ is composed by the maps $g$ of the form $(x,y)\mapsto (ax,by)$ such that $g^\ast \omega\wedge \omega =0$ which is equivalent to $g^*\omega=\theta\omega$ for some $\theta\in\mathbb{C}^*$. 
\begin{lemma} \label{calculo}Let $\omega$, $S$ and $G$ be as above.  Let $V(S)\subset \mathbb{C}^2$ be the algebraic set defined by $S$. Then there exists a bijection between $G$ and $V(S)\cap (\mathbb{C}^\ast)^2$. Precisely, the map  $(x,y)\mapsto (ax,by)$, $a,b\in\mathbb{C}^*$ belongs to $G$  if and only if $(a,b)$ belongs to $V(S)$. In particular $G$ is finite if and only if so is $V(S)\cap (\mathbb{C}^\ast)^2$. 
\end{lemma}
\begin{proof}
Let $g(x,y)=(ax,by)$ with $a, b\in\mathbb{C}^*$. If $g\in G$ there exists 
$\theta\in\mathbb{C}^*$ such that  $g^*\omega=\theta\omega$, that is: 
$$\sum\limits_{\alpha_{ij}\neq 0} a^{i+1}b^j\alpha_{ij} x^i y^j dx +
\sum\limits_{\beta_{ij}\neq 0}  a^i b^{j+1}\beta_{ij} x^i y^j dy=  \sum\limits_{\alpha_{ij}\neq 0} \theta\alpha_{ij} x^i y^j dx +
\sum\limits_{\beta_{ij}\neq 0}  \theta\beta_{ij} x^i y^j dy,    $$
\begin{equation}\label{eqtheta}
\sum\limits_{\alpha_{ij}\neq 0} (a^{i+1}b^j-\theta)\alpha_{ij} x^i y^j dx +
\sum\limits_{\beta_{ij}\neq 0}  (a^i b^{j+1}-\theta)\beta_{ij} x^i y^j dy= 0.
\end{equation}

 It follows that the numbers $a^{i+1}b^j$ for $\alpha_{ij}\neq 0$ and  $a^{i}b^{j+1}$ for $\beta_{ij}\neq 0$ are all equal. Therefore $(a,b)\in V(S)$. This proves $G\subset V(S)\cap (\mathbb{C}^\ast)^2$.
 
 Conversely, for any $(a,b)\in V(S)\cap (\mathbb{C}^\ast)^2$ we define the automorphism $g(x,y) = (ax,by)$. From the definition of $S$ we know that there exists $\theta \in \mathbb{C}^\ast$ such that the equation \ref{eqtheta} holds i.e. $g\in G$. Hence $V(S)\cap (\mathbb{C}^\ast)^2 \subset G$.
  \end{proof}
We now state and prove the main result of this section.
\begin{proposition} \label{pro1}Let $\omega$, $\mathcal{F}$ and $G$ be as in Lemma \ref{calculo} and suppose moreover that  $G$ is finite. Let $d$ be the degree of the extension of $\mathcal{F}$ to $\mathbb{P}^2$.  Then  
$$|G|\le d^2+d+1$$ and equality holds if and only if, up to the symmetry $(x,y)\mapsto (y,x)$,  the form $\omega$ is equal to
   $$\omega_{\alpha,\beta,\rho}=(\alpha + \rho x^dy)dx+(\beta y^{d}-\rho x^{d+1})dy$$ for some $\alpha,\beta,\rho\in\mathbb{C}^*$. 
  \end{proposition}

We need the following lemma, whose proof is given at the end of this section.
\begin{lemma}\label{bezout} Let $S$ be a finite set of polynomials in $\mathbb{C}[x,y]$ and set   $N=\max\limits_{f\in S}\deg f$.  Suppose that $V(S)$ is finite. Then 
$ |V(S)|\le N\deg f$ for all $f\in S$. 
\end{lemma}

\noindent\emph{Proof of Proposition \ref{pro1}.} 
 Suppose first that the line at infinity is invariant by $\mathcal{F}$. Let us prove that $|G|\le d^2+d$.
Since the line at infinity is invariant by $\mathcal{F}$, we have $\deg A, \deg B\le d$, hence the monomials in $\mathcal{M}$ have at most degree $d+1$ and therefore 
$\deg\tilde{f}\le d+1$ for all  $\tilde{f}\in S$ -- here $S$ is as defined in the beginning of this section. Thus, in view of Lemma \ref{bezout}, in order to prove that $|G|\le d^2+d$ it suffices to find some $\tilde{f}\in S$ with $\deg\tilde{f}\le d$. Observe that there exist three pairwise distinct monomials in $\mathcal{M}$; otherwise $S$ would contain only one element modulo sign  and $V(S)$ would be infinite. We can choose two between these three monomials, say $\mathfrak{m}_1$ and $\mathfrak{m}_2$, such that
both are divisible by $x$ or both are divisible by $y$. Then we have 
$\mathfrak{m}_1-\mathfrak{m}_2=x^ky^l\tilde{f}$ with $\tilde{f}\in S$, $k+l\ge 1$, so that $\deg\tilde{f}\le d$.

  Suppose now that the line at infinity is generically transverse to $\mathcal{F}$.
   Then the polynomial form $\omega$ defining $\mathcal{F}$ can be expressed as
$$\omega= \tilde{A}dx+\tilde{B}dy+yQ dx-xQdy,$$ 
where $\tilde{A}=0$ or $\deg\tilde{A}\le d$, $\tilde{B}=0$ or $\deg\tilde{B}\le d$,  and $Q\neq 0$ is homogeneous of degree $d$. Take two monomials $\mathfrak{m},\mathfrak{m}'\in \mathcal{M}$ with $f=\mathfrak{m}-\mathfrak{m}'\neq 0$. If these two monomials are associated to monomial forms appearing in
$\tilde{A}dx+\tilde{B}dy$, we have $\deg \mathfrak{m}, \deg\mathfrak{m}'\le d+1$ and therefore $\deg\tilde{f}\le\deg f\le d+1$. If $\mathfrak{m}$ or $\mathfrak{m}'$,  say $\mathfrak{m}$,
 is associated to a monomial form appearing in $yQ dx-xQdy$, then 
$\mathfrak{m}\in(xy)$. Then, since $\mathfrak{m}'$ is divisible by $x$ or $y$, we have  $f=\mathfrak{m}-\mathfrak{m}'=x^ky^l\tilde{f}$ with $k+l\ge 1$ and we obtain
again $\deg\tilde{f}\le d+1$. Therefore $\deg\tilde{f}\le d+1$ for all $\tilde{f}\in S$. If there exists $\tilde{f}\in S$ with $\deg\tilde{f}< d+1$, again by Lemma \ref{bezout} we obtain $|G|\le d^2+d$. Thus we can assume that $\deg\tilde{f}= d+1$
 for all $\tilde{f}\in S$.  Suppose that there exist two different monomials $\mathfrak{m}$ and $\mathfrak{m}'$ in $\mathcal{M}$  of degree $d+2$. Then these monomials are associated to monomial forms in $yQ dx-xQdy$, so  both monomials $\mathfrak{m}$ and $\mathfrak{m}'$ are divisible by $xy$ and therefore  $\mathfrak{m}-\mathfrak{m}'$ generates an element  in $S$ of degree $\le d$, which is a contradiction. We conclude that there exists a unique monomial in $\mathcal{M}$ of the form 
 $\mathfrak{m}=x^{p+1}y^{q+1}$, $p,q\ge 0$, $p+q=d$ and any other monomial in $\mathcal{M}$ has degree $\le d+1$.  As above, 
there exist at least two other monomials $\mathfrak{m}_1$ and $\mathfrak{m}_2$ different from $\mathfrak{m}$. Suppose that there exists another monomial $\mathfrak{m}_3\in\mathcal{M}$ different from $\mathfrak{m}$, $\mathfrak{m}_1$ and $\mathfrak{m}_2$. Then, there is a pair of monomials in $\{\mathfrak{m}_1,\mathfrak{m}_2,\mathfrak{m}_3\}$, say 
$\mathfrak{m}_1$ and $\mathfrak{m}_2$, that are both divisible by $x$ or both divisible by $y$.
 Then, since $\deg\mathfrak{m}_1,\deg\mathfrak{m}_2\le d+1$,  we have that
$\mathfrak{m}_1-\mathfrak{m}_2$ generates an element in $S$ of degree smaller than $d+1$, which is a contradiction. 
Then  $\mathcal{M}$ contains exactly two monomials other than $\mathfrak{m}$, say $\mathfrak{m}_1$ and $\mathfrak{m}_2$, which have degree at most $d+1$. If $\mathfrak{m}_1$ and $\mathfrak{m}_2$ have some common factor, 
again we have that  $\mathfrak{m}_1-\mathfrak{m}_2$ generates an element in $S$ of degree smaller than $d+1$; so we can assume that $\mathfrak{m}_1=x^{k_1}$,   $\mathfrak{m}_2=y^{k_2}$ for some $k_1,k_2\in\mathbb{N}$.  Then  $\tilde{f}=x^{k_1}-y^{k_2}$ is contained in $S$, so that  $\max\{k_1,k_2\}=\deg \tilde{f}=d+1$. Suppose that $k_2=d+1$; the 
other case is equal to the fist one up to the symmetry $(x,y)\mapsto (y,x)$. Since  $\mathfrak{m}-\mathfrak{m}_2=x^{p+1}y^{q+1}-y^{d+1}$ generates an element in $S$ of degree $d+1$, we deduce that  $q=0$, so that
$\mathfrak{m}=x^{d+1}y$. Now, since $\mathfrak{m}-\mathfrak{m}_1=x^{d+1}y-x^{k_1}$ generates an element of degree $d+1$ in $S$, we obtain  $k_1=1$ and therefore $\mathfrak{m}_1=x$. Then $\mathcal{M}=\{x^{d+1}y,x,y^{d+1}\}$ and therefore
 $$S=\{\pm(x^dy-1),\pm(x^{d+1}-y^d),\pm(x-y^{d+1})\}.$$
 It follows that
 $$V(S)= V(x^dy-1,x^{d+1}-y^d),$$ which gives \begin{equation}\label{igualdad}|G|=|V(S)|=d^2+d+1.\end{equation} Observe that the equality above only happens if the line at infinity is generically transverse to $\mathcal{F}$ and 
 $\mathcal{M}=\{x^{d+1}y,x,y^{d+1}\}$. From the construction of $\mathcal{M}$ it follows that $\tilde{A} = \alpha$, $\tilde{B} = \beta y^d$ and $Q = \rho x^d$ i.e
 $$\omega_{\alpha,\beta,\rho}=(\alpha + \rho x^dy)dx+(\beta y^{d}-\rho x^{d+1})dy,$$ where $\alpha,\beta,\rho\in\mathbb{C}^*$. Finally, it is easy to see  that the foliation defined by $\omega_{\alpha,\beta,\rho}$ has degree $d$ and  $d^2+d+1$ diagonal automorphisms. \qed\\

\noindent\emph{Proof of Proposition \ref{dg}.} Recall that a dihedral subgroup of $\operatorname{PGL}(2,\mathbb{C})$ is generated by a cyclic subgroup of index two which fixes two different points in $\mathbb{P}^1$ together with an automorphism of $\mathbb{P}^1$ permuting  these two points.  Then there are homogeneous coordinates such that 
$\aut (\F)$ is generated by a group of diagonal transformations
  together with a map of the form  $$R\colon[x:y:z]\mapsto [\mu y:\nu x:z],$$ where $\mu,\nu\in\mathbb{C}^*$. By Proposition \ref{pro1}, the subgroup $G$ of diagonal transformations in $\aut (\F)$ satisfies $|G|\le d^2+d+1$. Thus, since  $[\aut (\F):G]=2$, we obtain $$|\aut (\F)|\le 2(d^2+d+1).$$ \qed

We finish this section with a lemma which will be used later in the proof of Proposition \ref{nd}.
\begin{lemma}\label{aux} Let $\mathcal{F}$, $\mathcal{M}$ and $G$ be as defined at the beginning of this section.   Suppose that
\begin{enumerate}
\item $\mathcal{F}$ has degree $d\ge 2$;
\item the line at infinity is invariant by $\mathcal{F}$; and
\item  $\mathcal{M}$ contains  $M=\{xy,x^dy,xy^d\}$.
\end{enumerate}
Then $|G|= (d-1)^2$ if  $\mathcal{M}=M$ and
$|G|<(d^2+d+1)/2$ otherwise.
\end{lemma}
\begin{proof}
Suppose $\mathcal{M}=M$. Then $$S=\{\pm(x^{d-1}-1),\pm(y^{d-1}-1),\pm(x^{d-1}-y^{d-1})\},$$
so that $$|G|=\# V(x^{d-1}-1,y^{d-1}-1)=(d-1)^2.$$
Assume that $\mathcal{M}$ contains some monomial $\mathfrak{m}\notin M$. Suppose first that $\mathfrak{m}=x^k$ for some positive integer $k$.
If $k=1$, we obtain that  $y-1$ belongs to $S$ and therefore 
$$|G|\le \# V \left( y-1,x^{d-1}-1 \right)=d-1<\frac{d^2+d+1}{2}.$$ 
If $k\in\{d,d+1\}$, then $x^k-x^dy=x^d(x^{\xi}-y)$ for $\xi\in\{0,1\}$, hence $x^{\xi}-y$ belongs to $S$ and therefore
$$|G|\le \# V \left( x^{\xi}-y,x^{d-1}-1 \right)=d-1<\frac{d^2+d+1}{2}.$$ 
Thus we may assume $1<k<d$. We have $x^k-xy=x(x^{k-1}-y)$, so that $x^{k-1}-y$ belongs to $S$ and therefore
\begin{align} |G|\le \#V\left(x^{k-1}-y, y^{d-1}-1\right) 
\label{upa} =  (k-1)(d-1).
\end{align} 
 On the other hand, $x^dy-x^k=x^k(x^{d-k}y-1)$, so that $x^{d-k}y-1$ belongs to $S$ and therefore 
\begin{align}\label{uta}|G|\le \#V\left(x^{d-k}y-1, y^{d-1}-1\right)=(d-k)(d-1).\end{align} If we sum the inequalities \eqref{upa}
and \eqref{uta} we obtain $$|G|\le \frac{(k-1)(d-1)+(d-k)(d-1)}{2}=\frac{(d-1)^2}{2}<\frac{d^2+d+1}{2}.$$ 
If $\mathfrak{m}=y^k$ for some $k\in\mathbb{N}$, the proof follows as above. Thus, we can assume that 
$\mathfrak{m}=x^{k}y^{l}$ with $k,l\ge 1$, $k+l\le d+1$. We can suppose $k\ge l$; the other case is quite similar. Then, since $\mathfrak{m}\neq xy$, we have $k>1$. Then $x^{k}y^{l}-xy=(x^{k-1}y^{l-1}-1)$, so  $x^{k-1}y^{l-1}-1$ belongs to $S$ and therefore 
\begin{align} |G|\le \#V\left(x^{k-1}y^{l-1}-1, y^{d-1}-1\right) 
\label{upa1} =  (k-1)(d-1).
\end{align} Since $\mathfrak{m}\neq x^d y$, we have $k<d$. Then 
$x^dy-x^{k}y^{l}=x^ky(x^{d-k}-y^{l-1})$, so  $x^{d-k}-y^{l-1}$ belongs to $S$ and therefore
\begin{align}\label{uta1}|G|\le \#V\left(x^{d-k}-y^{l-1}, y^{d-1}-1\right)=(d-k)(d-1).\end{align}Finally, as above the result follows
by summing \eqref{upa1} and \eqref{uta1}.

\end{proof}

\noindent\emph{Proof of Lemma \ref{bezout}.}  By Bézout's Theorem the lemma holds for $|S|\le 2$.  Assume that Lemma holds if $|S|\le k\in\mathbb{N}$. Suppose now that $|S|=k+1$ and fix any $f\in S$. Write $$S=\{f_1,\ldots, f_{k+1}\}$$ with $f_1=f$. We can put $f_1=hg_1,\ldots, f_k=hg_k$, where $h$ is the greatest common divisor of the $f_j$; so the set $V(g_1,\ldots,g_k)$ is finite and by the inductive hypothesis 
$$|V(g_1,\ldots,g_k)|\le N\deg g_1.$$ Observe that $$V(S)\subset V(h,f_{k+1})\cup V(g_1,\ldots,g_k).$$
Clearly $h$ and $f_{k+1}$ have no common factor, otherwise $V(S)$ would be infinite. Then, by Bézout's Theorem we have $$ |V(h,f_{k+1})|\le  \deg f_{k+1}\deg h\le N\deg h$$ and therefore
 \begin{align*}|V(S)|\le |V(h,f_{k+1})| + |V(g_1,\ldots,g_k)|\\ 
 \le N \deg h +N\deg g_1= N\deg f.
 \end{align*}   \qed

\section{Transitive imprimitive groups}\label{stig}
This section is devoted to prove Proposition \ref{tig1} and Proposition \ref{nd}.
We need the following lemma whose proof is easily obtained and so we omit it.  
\begin{lemma} \label{mon}Consider the polynomial 1-form with isolated singularities on $\mathbb{C}^2$ given by 
$$\omega=\sum\limits_{i,j\ge 0} \alpha_{ij} x^i y^j dx +
\sum\limits_{i,j\ge 0}  \beta_{ij} x^i y^j dy,$$ where $\alpha_{ij},\beta_{ij}\in\mathbb{C}$. Let $\mathcal{F}$ be the holomorphic foliation defined by $\omega$ on $\mathbb{P}^2$ and suppose that $\mathcal{F}$ is invariant by the transformation 
$$R\colon[x:y:z]\mapsto [\mu y:\nu x:z],$$ where $\mu,\nu\in\mathbb{C}^*$.
  Then,  $\alpha_{ij}\neq 0$  if and only if   $\beta_{ji}\neq 0$.
\end{lemma}

\noindent\emph{Proof of Proposition \ref{tig1}.} Let $G<\aut(\F)$ be the subgroup of transformations of the form $$[x:y:z]\mapsto [ax:by:z].$$ It is not difficult to see that $[\aut(\F):G]=3$. Then  by Proposition \ref{pro1} we obtain $$|\aut(\F)|=3 |G| \le 3(d^2+d+1).$$ Suppose that the inequality above is an equality. Then, by Proposition \ref{pro1} we have that $\F$
is defined by the 1-form 
 $$\omega=(\alpha + \rho x^dy)dx+(\beta y^{d}-\rho x^{d+1})dy$$ for some $\alpha,\beta,\rho\in\mathbb{C}^*$. Since $\F$ is invariant by the transformation $$ [x:y:z]\mapsto [y:z:x],$$
 by substituting $(x,y)\leftarrow (y/x,1/x)$ in $\omega$ we obtain a meromorphic 1-form $\tilde{\omega}$ such that
   $x^{d+2}\tilde{\omega}=\theta\omega$ for some $\theta\in\mathbb{C}^*$. By a direct computation we obtain
   \begin{align*}x^{d+2}\tilde{\omega}&= x^{d+2}\left(\left[\alpha + \rho (y/x)^d\left(1/x\right)\right]\frac{xdy-ydx}{x^2}-\left[\beta (1/x)^{d}-\rho (y/x)^{d+1}\right]\frac{dx}{x^2}\right)\\
 &=(-\beta - \alpha x^dy)dx+(\rho y^{d}+\alpha x^{d+1})dy,
 \end{align*}
hence
 $-\beta=\theta\alpha$, $-\alpha=\theta\rho$, $\rho=\theta\beta$ and $\alpha=-\theta\rho$ and therefore $\theta^3=1$, $\alpha=-\theta\rho$ and $\beta=\theta^2\rho$. From this we conclude that $\mathcal{F}$ is defined by the 1-form 
$$\omega_\theta= (x^dy-\theta)dx+(\theta^2 y^{d}-x^{d+1})dy,$$ where $\theta^3=1$. Take $a,b\in\mathbb{C}^*$ such that $b^{d^2+d+1}=\theta^{1-d}$ and $a=\theta b^{d+1}$, and consider the map $f(x,y)=(ax,by)$. A direct computation shows that  
\begin{align*}f^*(\omega_\theta)&=(a^{d+1}bx^dy-\theta a)dx+(\theta^2b^{d+1} y^{d}-a^{d+1}bx^{d+1})dy\\
&=\theta^2b^{d+1}\left[(x^dy-1)dx+(y^{d}-x^{d+1})dy\right],
\end{align*} so the  foliation $\mathcal{F}$ is equivalent to the Jouanolou Foliation of degree $d$. \qed \\

\noindent\emph{Proof of Proposition \ref{nd}.}   
Let $\omega=Adx+Bdy$ be a polynomial 1-form with isolated singularities defining $\mathcal{F}$ on $\mathbb{C}^2=\{[x:y:1]\colon x,y\in\mathbb{C}\}$.
 Let $G< \aut(\F)$ be the subgroup of transformations of the form $$[x:y:z]\mapsto [ax:by:z].$$ It is not difficult to see that $[\aut(\F):G]=6$. 
Suppose that the line at infinity $z=0$ is invariant by $\mathcal{F}$. 
   Since the three lines $x=0$, $y=0$ and $z=0$ are permuted by $$T\colon [x:y:z]\mapsto [y:z:x] \text{ and } R\colon[x:y:z]\mapsto [\mu y:\nu x:z] $$ which leave $\mathcal{F}$ invariant, 
   then the three lines $x=0$, $y=0$ and $z=0$ are invariant by $\mathcal{F}$.
    The invariance of the line $y=0$ means that $A(x,0)=0$. Thus, since $\omega$ has isolated singularities the polynomial $B(x,0)$ is not zero, so we have 
     \begin{equation}\label{xdy}B(x,y)= c_kx^k+\cdots + c_l x^l +O(y),
     \end{equation} where $0\le k\le l\le d$, $c_k\neq 0$, $c_l\neq 0$.
      If we denote $L=\{y=0\}$, $0:=[0:0:1]$ and $\infty=[1:0:0]$, from equation \eqref{xdy} we have that
     $$\operatorname{Z}(\mathcal{F},L, 0)=k,\; \operatorname{Z}(\mathcal{F},L,\infty)= d+1-l, $$ where $\operatorname{Z}$ denote the index defined in in \cite[p.15]{Bru} which coincides the Gomez-Mont-Seade-Verjovsky index \cite{gsv}. Indeed $\operatorname{Z}(\mathcal{F},L, 0)$ is the vanishing order of $B(x,0)$ and $\operatorname{Z}(\mathcal{F},L,\infty)$ is the vanishing order of 
     $$
     z^{d+1}B\left(\frac{1}{z},0\right) = c_kz^{d+1-k}+\cdots + c_l z^{d+1-l}. 
     $$
     Observe that $R\circ T$ leave $L$ invariant and maps $0$ to $\infty$. Thus, since $\mathcal{F}$ is invariant by $R\circ T$ we have that $$\operatorname{Z}(\mathcal{F},L, 0)=\operatorname{Z}(\mathcal{F},L,\infty),$$ hence 
     \begin{equation}\label{kl}k+l= d+1,\end{equation} and therefore \begin{equation}\label{km}k\le \frac{d+1}{2}.\end{equation} 
    From \eqref{xdy} 
    we see that  the monomial $x^kdy$ appear in $\omega$ with a nonzero coefficient and, by Lemma \ref{mon},
    the same happens with $y^kdx$. Then the monomials $x^ky$ and $xy^k$ are contained in $\mathcal{M}$ and therefore
     the difference $x^ky-xy^k$ generates the element $x^{k-1}-y^{k-1}$  in $S$ if $k>1$, or the element $y-x$ in $S$ if $k=0$. Anyway, if we assume $k\neq 1$, the monomial $x^{|k-1|}-y^{|k-1|}$ belongs to $S$; we leave the case $k=1$ for later. 
        It follows from Lemma \ref{calculo}  that for all transformation $g(x,y)=(ax,by)$  in $G$ we have that  
      \begin{equation}\label{eq3}a^{|k-1|}=b^{|k-1|}.\end{equation} Moreover, if $g\in G$ then   $$T\circ g\circ T^{-1}\colon(x,y)\mapsto \left(\frac{b}{a}x,\frac{1}{a}y\right)$$ also belongs to $G$,
      so by equation \eqref{eq3} we have $$(b/a)^{|k-1|}=(1/a)^{|k-1|}$$ and therefore
      \begin{equation}\label{eq4}a^{|k-1|}=b^{|k-1|}=1,\end{equation} for all $g\in G$.  It follows from these equations that $|G|\le |k-1|\cdot|k-1|$ and together with inequality \eqref{km} we obtain  
     $$|\aut(\F)|\le 6|k-1|^2\le 6\left(\frac{d-1}{2}\right)^2<3(d^2+d+1).$$ 
   Suppose now that $k=1$. Then  by equation \eqref{kl}
   we have  $l=d$.  From  equation \eqref{xdy}  the monomial $xdy$ and $x^ddy$ appear in $\omega$ with a nonzero coefficient and, by Lemma \ref{mon}, the same happens for the monomials $ydx$ and $y^ddx$. Then the monomials $xy$, $x^dy$ and $xy^d$ are contained in $\mathcal{M}$. If $\mathcal{M}$ contains some other monomial,  by Lemma \ref{aux} we have 
   $$|G|<\frac{d^2+d+1}{2},$$ so that $$|\aut(\F)|<3(d^2+d+1).$$
   Then assume that \begin{equation}\label{mm}\mathcal{M}=\{xy,x^dy,xy^d\}.\end{equation}
    Therefore, again by Lemma \ref{aux} we obtain $$|\aut(\F)|=6(d-1)^2.$$
   We shall determine the possibilities for $\omega$. From equation \eqref{mm} we see that, besides $xdy$, $x^ddy$, $ydx$ and $y^ddx$, the form  $\omega$ could contain the monomials $x^{d-1}ydx$ and $xy^{d-1}dy$, so that, up to multiplication by a nonzero complex number, we can write
   $$\omega=ydx+py^ddx+qx^{d-1}ydx+axdy+bx^ddy+cxy^{d-1}dy,$$
   where $p,q,a,b,c\in\mathbb{C}$ and $p,a,b\in\mathbb{C}^*$.
   Since $\mathcal{F}$ is invariant by the transformation $T$, by substituting $(x,y)\leftarrow (y/x,1/x)$ in $\omega$ we obtain a form $\tilde{\omega}$ such that
   $x^{d+2}\tilde{\omega}=\theta\omega$ for some $\theta\in\mathbb{C}^*$. By a direct computation we obtain
   $$x^{d+2}\tilde{\omega}= (-p-c)ydx+(-q-b)y^ddx+(-a-1)x^{d-1}ydx+pxdy+x^ddy+qxy^{d-1}dy,$$
  so we have 
   \begin{equation}\label{sistema}-p-c=\theta,\; -q-b=\theta p,\; -a-1=\theta q,\; p=\theta a,\; 1=\theta b,\; q=\theta c.\end{equation}
  By substituting $q=-\theta^{-1}(a+1)$, $p=\theta a$ and  $b=\theta^{-1}$ in the second equation above we have
  $$\theta^{-1}(a+1)-\theta^{-1}=\theta^2 a,$$  so that $\theta^3=1$. Then we find that the equations \eqref{sistema} are equivalent to  $$q=-\theta^{2}(a+1),\;p=\theta a,\; b=\theta^{2},\;c=-\theta(a+1), \; a\in\mathbb{C}^*,$$ so that $\omega$ is given by
 $$\omega_{a,\theta}=ydx+\theta ay^ddx-\theta^2(a+1)x^{d-1}ydx
 +axdy+\theta^2x^ddy-\theta(a+1)xy^{d-1}dy.$$
Since $\mathcal{F}$ is invariant by  $R$, if we do the substitution $(x,y)\leftarrow (\mu y,\nu x)$ in $\omega_{a,\theta}$ we must obtain 
a form $\tilde{\omega}$ such that $\tilde{\omega}=\lambda\omega_{a,\theta}$ for some $\lambda\in\mathbb{C}^*$.  A simple computation gives
\begin{align*}\tilde{\omega}=\nu\mu xdy+\nu^d\mu\theta ax^ddy-\mu^d\nu\theta^2(a+1)y^{d-1}xdy&\\
+\nu\mu aydx+\nu\mu^d&\theta^2y^ddx-\nu^d\mu\theta(a+1)yx^{d-1}dx,\end{align*}

so that \begin{align} \label{sis}\nu\mu=\lambda a, \;   \nu^d\mu\theta a= \lambda \theta^2    , \;   -\mu^d\nu\theta^2(a+1)&= -\lambda\theta(a+1)   ,  \\  
\nonumber\nu\mu a= \lambda , \;   \nu\mu^d \theta^2= \lambda\theta a    , &  \; -\nu^d\mu\theta(a+1)= -\lambda\theta^2(a+1). 
\end{align}
from $\nu\mu=\lambda a$ and $\nu\mu a= \lambda$ we obtain $a=\pm 1$. Suppose first that $a=-1$. Then the system \eqref{sis} becomes
$$\nu\mu=-\lambda,\;  \nu^{d-1} = \theta   ,\; \mu^{d-1} =\theta^2.$$ Thus, in this case the form $\omega$ effectively
exists and is given by 
$$\omega_{-1,\theta}=ydx-\theta y^ddx
 -xdy+\theta^2x^ddy.$$ By substituting $(x,y)\leftarrow (\alpha x, \beta y)$, $\alpha^{d-1}=\theta$, 
 $\beta^{d-1}=\theta^2$ in the form above we obtain
 the form
 $$\alpha\beta\left[ydx-y^ddx
 -xdy+x^ddy\right]$$ and therefore $\mathcal{F}$ is isomorphic to the foliation $\mathcal{F}_d$.
 Suppose now that $a=1$. Then the system \eqref{sis} becomes 
$$\nu\mu=\lambda,\;  \nu^{d-1} =\theta   ,\; \mu^{d-1} =\theta^2$$ and we obtain
 $$\omega_{1,\theta}=ydx+\theta y^ddx-2\theta^2 x^{d-1}ydx
 +xdy+\theta^2x^ddy-2\theta xy^{d-1}dy.$$
Again, by substituting $(x,y)\leftarrow (\alpha x, \beta y)$, $\alpha^{d-1}=\theta$, 
 $\beta^{d-1}=\theta^2$ we obtain
 the form
   $$\alpha\beta \left[ydx+ y^ddx-2 x^{d-1}ydx
 +xdy+x^ddy-2xy^{d-1}dy\right],$$ so that $\mathcal{F}$ is isomorphic to $\mathcal{G}_d$. This finishes the 
 proof of the part $(1)$ of Proposition \ref{nd}.

     Now, assume that  the line at infinity $L_\infty = \{z=0\}$ is generically transverse to $\mathcal{F}$. Then the line $y=0$ is also generically transverse to $\mathcal{F}$, so that $A(x,0)$ is not zero. Therefore 
        $$A(x,0)=c_kx^k+O(x^{k+1})$$ with $c_k\neq 0$, $k\ge 0$,  so that  $\omega$ contains the monomial $x^kdx$ and,
          by Lemma \ref{mon}, also the monomial  $y^kdy$.
        It follows from Proposition \ref{pro1}  that for all transformation $g(x,y)=(ax,by)$  in $G$ we have 
      \begin{equation}\label{eq6}a^{k+1}=b^{k+1}.\end{equation} Thus, since  $TgT^{-1}\colon (x,y)\mapsto \left(\frac{b}{a}x,\frac{1}{a}y\right)$ also belongs to $G$ we have $$(b/a)^{k+1}=(1/a)^{k+1}$$ and therefore
      \begin{equation}\label{eq7}a^{k+1}=b^{k+1}=1.\end{equation} From these equations we obtain $|G|\le (k+1)(k+1)$  and consequently \begin{equation}\label{6k}|\aut(\F)|\le 6(k+1)^2.\end{equation} 
      Observe that $k$ is the  tangency order of $\mathcal{F}$ to $L=\{y=0\}$ at
       $0=[0:0:1]$. Since $R\circ T$ leave $L$ invariant and maps $0$ to $\infty=[1:0:0]$, we have
      \begin{equation}\label{tang}\operatorname{Tang} (\mathcal{F},L, 0)=\operatorname{Tang} (\mathcal{F},L,\infty)=k. \end{equation} Thus, since the total tangency of $\mathcal{F}$ to the line $L$ is exactly $d$, we conclude that $k\le d/2$. Then, from Equation \ref{6k} we have 
      \begin{equation}\nonumber|\aut(\F)|\le 6\left(\frac{d}{2}+1\right)^2,\end{equation}  
      so we obtain \begin{equation}\label{6kd}|\aut(\F)|< 3(d^2+d+1)\end{equation} for all $d>2$. Thus, to continue with the remaining case we assume $d=2$. Then we have $k\in\{0,1\}$.  If $k=0$, Equation \eqref{6k} gives
       $|\aut(\F)|=6<3(d^2+d+1)$. Then we may assume $k=1$ and therefore $\omega$ contains the monomials $xdx$ and  $ydy$ with nonzero coefficients.  Since $\F$ is generically transverse to the line at infinity $L_\infty$, the homogeneous part of maximum degree of $\omega$ has the form 
       $$Q(ydx-xdy),$$ where $Q\in\mathbb{C}[x,y]$ is homogeneous of degree $2$. From  \eqref{tang} and the symmetries of $\F$ we deduce that $\F$ has nonzero tangency orders with $L_\infty$ at the points $[0:1:0]$ and $[1:0:0]$. This means that $x$ and $y$ are factors of $Q$, so we have that $Q=cxy$ for some $c\in\mathbb{C}^*$. Clearly we can assume that $c=1$. Then $\omega$ contains the monomials $xy^2dx$ and $-x^2ydy$. We deduce that the set of monomials $\mathcal{M}$ -- as defined in Section \ref{sec2} -- contains 
       $$M= \{x^2, y^2, x^2y^2\}.$$ Let us prove that $\mathcal{M}\neq M$ implies $|\aut(\F)|<3(d^2+d+1)$. Observe that $\mathcal{M}\backslash M$ can only contain monomials of degrees $2$ or $3$, so we have 
       $$\mathcal{M}\backslash  M\subset\{xy,x^3,y^3,x^2y,xy^2\}.$$ For example, suppose that $xy\in\mathcal{M}$. Then $xy-y^2$ and $x^2y^2-y^2$ generate the elements $x-y$ and $x^2-1$ in the set $S$ -- as defined in Section \ref{sec2}. Then $$|G|\le \# V\left(x-y,x^2-1\right)= 2,$$ so that 
       $$|\aut(\F)|\le 12<3(d^2+d+1).$$ The analysis for each of the other possible monomials in $\mathcal{M}\backslash M$ is quite similar. 
        Then we can assume $$\mathcal{M}= \{x^2, y^2, x^2y^2\}$$
    and therefore $\omega$ has the form
       $$\omega=\left(ax+xy^2\right)dx+\left(by-x^2y\right)dy,$$ where $a,b\in\mathbb{C}^*$. 
 Since $\mathcal{F}$ is invariant by the transformation $T$, by substituting $(x,y)\leftarrow (y/x,1/x)$ in $\omega$ we obtain a form $\tilde{\omega}$ such that
   $x^{4}\tilde{\omega}=\theta\omega$ for some $\theta\in\mathbb{C}^*$. By a direct computation we have
   $$x^{4}\tilde{\omega}= \left(-bx-axy^2\right)dx+\left(y+ax^2y\right)dy,$$
   so that $-b=\theta a$, $-a=\theta $, $1=\theta b$ and we obtain obtain $\theta^3=1$, $a=-\theta$, $b=\theta^2$. Then 
  $$\omega=\left(-\theta x+xy^2\right)dx+\left(\theta^2y-x^2y\right)dy.$$
  By substituting $(x,y)\leftarrow (\alpha x, \beta y)$, $\alpha^{2}=\theta^2$, 
 $\beta^{2}=\theta$ in the form above we obtain
 
 $$\left(-x+xy^2\right)dx+\left(y-x^2y\right)dy$$ and therefore $\mathcal{F}$ is isomorphic to the foliation $\mathcal{S}$. \qed
      
 \section{Polyhedral groups}\label{poly}
 
 This section is devoted to prove Proposition \ref{polyhedral}. We begin with some definitions. Let $F$ be a complex polynomial in two variables, let  $\omega$ be a polynomial 1-form in two variables and let $G$ be a subgroup of $\operatorname{GL}(2,\mathbb{C})$. We say that $F$ is semi-invariant by $g\in\operatorname{GL}(2,\mathbb{C})$ if there exist a constant $\lambda\in\mathbb{C}^*$ such that $F(g):= F\circ g =\lambda F$. Similarly,  the 1-form $\omega$ is semi-invariant by $g\in\operatorname{GL}(2,\mathbb{C})$ if there exist a constant $\lambda \in\mathbb{C}^*$ such that $g^*(\omega)=\lambda \omega$.  We say that $F$  (resp. $\omega$) is $G$-semi-invariant if $F$  (resp. $\omega$) is semi-invariant by every $g\in G$. We  say that the 1-form $\omega\neq 0$ is   homogeneous of degree $n$, if we have  $\omega=Adx+Bdy$, where $A$ (resp. $B$) is a homogeneous polynomial of degree $n$ or $A=0$ (resp. $B=0$). \\
 
As we have seen in Subsection \ref{gg}, in suitable homogeneous coordinates $(X:Y:Z)$ the group $\aut (\mathcal{F})$ preserves $\{Z=0\}$ hence it corresponds finite non abelian subgroup $G$ of $\operatorname{GL}(2,\mathbb{C})$ acting on the affine chart $\{Z\neq 0\}$ with natural coordinates $(x,y)$. Moreover we have that its image $\overline{G}<\operatorname{PGL}(2,\mathbb{C})$ is a polyhedral group.
  The foliation $\F$ is defined by a $G$-semi-invariant polynomial 1-form $\omega$, which can be expressed as a sum  $$\omega=\sum_{j\in \mathcal{J}}\omega_j,$$ where $\mathcal{J}\subset \{0,\ldots, d+1\}$ and each $\omega_j$ is homogeneous of degree $j$.
Observe
  that the origin $0\in\mathbb{C}^2$ need to be a singularity of $\mathcal{F}$, otherwise the tangent line to $\mathcal{F}$ at the origin would define a fixed point of $\overline{G}$, 
  which is impossible because $\overline{G}$ is a polyhedral group. Then we have  $\mathcal{J}\subset \{1,\ldots, d+1\}$. Moreover, notice that  $|\mathcal{J}|\ge 2$, otherwise $\omega$ would be homogeneous and $\aut (\F)$ would be infinite. 
  Given any $g\in G$,  there exists a constant $\lambda \in\mathbb{C}^*$ such that  
 $$\sum\limits_{j\in\mathcal{J}}g^*(\omega_j)= \lambda \sum\limits_{j\in\mathcal{J}}\omega_j.$$ Then, since $g$ is linear we necessarily have $$ g^*(\omega_j)= \lambda \omega_j$$ for all $j\in\mathcal{J}$, from which we conclude that each $\omega_j$ is $G$-semi-invariant.

 \begin{lemma}\label{div2} 
	Let $\omega\neq 0$ be a homogeneous 1-form of degree $n$ in two variables. Then we can express
	\[
	\omega = dP + Q(xdy-ydx),
	\]  
	where $P$ and $Q$ are homogeneous of degrees $n+1$ and $n-1$ respectively. Moreover if $\omega$ is ${G}$-semi-invariant for some $G<\gl(2,\mathbb{C})$ then so are $P$ and $Q$. 
	
\end{lemma}

\begin{proof} Let $\omega=Adx+Bdy$ and 
	define $P = \frac{1}{n+1}(Ax+By)$ and $Q=\frac{1}{n+1}\left(\frac{\partial B}{\partial x}-\frac{\partial A}{\partial y}\right)$. A straightforward computation -- we only need to use the relations $nA=x\frac{\partial A}{\partial x}+y\frac{\partial A}{\partial y}$ and $nB= x\frac{\partial B}{\partial x}+y\frac{\partial B}{\partial y}$ -- shows that $\omega=dP+Q(xdy-ydx)$.  It remains to prove that $P$ and $Q$ are $G$-semi-invariant if so is $\omega$. Let $g(x,y) = (ax+by,cx+dy) \in {G}$. Since $g^*(\omega)=\lambda \omega$ for some $\lambda \in\mathbb{C}^*$, we have $$A(g) d(ax+by)+B(g) d(cx+dy)=\lambda Adx+\lambda Bdy,$$ so that 
	\begin{align*}
	A(g) a+B( g)c &= \lambda A,\\
	 A(g)b+B(g)d &= \lambda B.
	\end{align*}
Then $$P(g) =\frac{1}{n+1}\left[A(g) (ax+by)+B(g) (cx+dy)\right]=\lambda P$$ and therefore $P$ is semi-invariant by $g$.  From the last equation we obtain $$g^*(dP)=d\left( P(g)\right)=\lambda dP,$$ so $dP$ is semi-invariant by $g$.  
Then $Q(xdy-ydx)=\omega-dP$ is semi-invariant by $g$ and from this we easily obtain that $Q$ is semi-invariant by $g$. 
\end{proof}
 
Now we continue with the proof of Proposition \ref{polyhedral}.  The group $\overline{G}$ is the quotient of $G$ by its subgroup $H$ of homotheties. We will relate the order of $H$ with the degree of the foliation.  If $|H|=n$, we have that $H$  is generated by an element of the form $h=\xi I$, where $\xi$ is a primitive $nth$-root of unity.  Since $\omega$ is semi-invariant by $h$ there exist $c\in\mathbb{C}^*$ such that
 $\sum h^*(\omega_j)= c\sum\omega_j,$ that is
 $$\sum\limits_{j\in\mathcal{J}} \xi^{j+1}\omega_j= c\sum\limits_{j\in\mathcal{J}} \omega_j.$$ Then $\xi^{j+1}=c$ for all $j\in\mathcal{J}$.  Thus, if $j_1 <j_2$ are elements in $\mathcal{J}$,  we have $\xi^{j_2-j_1}=1$ and therefore $n\le j_2-j_1$. We conclude 
 that
 \begin{equation}\label{reto} |H|\le |j_2-j_1|,\;\forall j_1, j_2 \in \mathcal{J} \textrm{ distinct}. 
 \end{equation}
 In particular we obtain  $$|H|\le d$$
 that together with our assumption $3(d^2+d+1) \le \left|G\right|$ leads to
 \begin{equation}\label{ineqp}
3(d^2+d+1) \le \left|G\right|= \left|\overline{G}\right||H|\le \left|\overline{G}\right| d.
\end{equation}

According to Lemma \ref{div2} we can express $\omega_j=dP_j+Q_j(xdy-ydx)$ for each 
$j\in\mathcal{J}$.   Notice the following facts:
\begin{enumerate}
\item \label{fact1}If $d+1\in\mathcal{J}$, then $P_{d+1}$ is zero: if $\omega_{d+1}\neq 0$,  the foliation is generically transverse to the line at infinity and we have $\omega_{d+1}=Q(xdy-ydx)$ for some homogeneous polynomial $Q$. If this were the case, we would have
$$dP_{d+1}+Q_{d+1}(xdy-ydx)=Q(xdy-ydx)$$ and it is easy to see that this would imply that $P_{d+1}=0$.
\item The polynomial  $P_{j}$ is not zero for some $j\in\mathcal{J}$. Otherwise $\F$ would be the radial foliation $xdy-ydx=0$ and $\aut(\F)$ would be infinite.
\end{enumerate}

First recall that for any finite subgroup $\overline{G} < \pgl(2,\mathbb{C})$ the stabilizer $\overline{G}_p$ of any point $p\in\mathbb{P}^1$ is cyclic. Then the smallest orbits of such group correspond to its elements with largest order. We refer to Klein's book \cite{K} for generalities on the polyhedral groups.
 
  \noindent Case 1: suppose that $\overline{G}=T$ the tetrahedral group. Since $|T|=12$, Inequality \eqref{ineqp} implies that $d=2$. Take $j'\in\mathcal{J}$ such $P_{j'}\neq 0$. By the fact \eqref{fact1} above we have $j'\le d=2$, so 
  that $\deg P_{j'}\le 3$. Then $P_{j'}=0$ defines a set of order $\le 3$ in $\mathbb{P}^1$ 
   which is invariant  by the action of $T$; but this is a contradiction because the smallest orbit of $T$
    in $\mathbb{P}^1$ has order $4$. \\

  \noindent Case 2: suppose that $\overline{G}=O$ the octahedral group.  The smallest orbit of $O$ has order $6$ and it is well known that, up to  linear change of coordinates, we can assume that this orbit corresponds to the 6 lines defined by the polynomial 
  $$P=xy(x^4-y^4).$$  Since $|O|=24$, Inequality \eqref{ineqp} implies that $d\le 6$. 
  Take $\omega_{j'}=dP_{j'}+Q_{j'}(xdy-ydx)$ with $P_{j'}\neq 0$.  By the fact \eqref{fact1} above we have $j'\le d\le 6$, so that $\deg P_{j'}\le 7$.  Thus, since $P_{j'}=0$ defines a finite 
  set in $\mathbb{P}^1$ which is invariant by $O$ and the smallest orbits of $O$ have orders $6$ and $8$, we conclude that $\deg P_{j'}=6$, ${j'}=5$  and 
  any other $P_j$ is zero. Moreover, since $P_5$ has degree 6 and $P_5=0$ has to be an equation of the 6 lines composing the orbit of order $6$ 
  of $O$,  we have  $P_5= \beta P$ for some 
  $\beta\in\mathbb{C}^*$. Observe that $Q_5=0$, otherwise $Q_5=0$ would define a set of order $<6$ which would be invariant by $O$. Then $\omega_5=\beta dP$, hence $d\ge5$ and from Inequality \eqref{ineqp} we 
  obtain that $|H|\ge 4$. Then, from Inequality \eqref{reto} we deduce that the only other $\omega_j\neq 0$ that can exist -- and actually exists because $\omega$ is not homogeneous -- is $\omega_1$. Since $P_1$ is zero, we have $\omega_1=\alpha(xdy-ydx)$ for some $\alpha\in\mathbb{C}^*$ and therefore $$\omega=\alpha(xdy-ydx)+\beta dP.$$ Applying the substitution
   $(x,y)\leftarrow (\lambda x, \lambda y)$ with $\lambda^4=\alpha/\beta$ we obtain the 1-form
    $$\alpha\lambda^2\left(xdy-ydx+dP\right),$$ so that $\F=\mathcal{P}_5$. \\

\noindent Case 3: suppose that $\overline{G}=I$ the icosahedral group. 
The smallest orbit of $O$ has order $12$ and it is well known that, up to linear change of coordinates, we can assume that this orbit corresponds to the 12 lines defined by the polynomial 
  $$P=xy(x^{10}+11x^5y^5-y^{10}).$$  Since $|I|=60$, Inequality \eqref{ineqp} implies that $d\le 18$. 
  Take $\omega_{j'}=dP_{j'}+Q_{j'}(xdy-ydx)$ with $P_{j'}\neq 0$.   By the fact \eqref{fact1} above we have $j'\le d\le 18$, so that $\deg P_{j'}\le 19$.  Thus, since $P_{j'}=0$ defines a finite 
  set in $\mathbb{P}^1$ which is invariant by $I$ and the smallest orbits of $I$ have orders $12$ and $20$, we conclude that $\deg P_{j'}=12$,  ${j'}=11$  and 
  any other $P_j$ is zero. Moreover, since $P_{11}$ has degree 12 and $P_{11}=0$ is an equation of the 12 lines composing the orbit of order $12$ 
  of $I$,  we have  $P_{11}= \beta P$ for some 
  $\beta\in\mathbb{C}^*$. Recall that $Q_j=0$ defines a set of order $\deg Q_j=j-1\le d\le 18$ in $\mathbb{P}^1$, which
   is invariant by $I$. Thus, since the smallest orbits of $I$ have orders 12 and 20, we necessarily have 
    $\deg Q_j\in\{0,12\}$. If $\deg(Q_j)=12$, we have that $\omega_{13}\neq 0$ and,
   since $\omega_{11}$ is also nonzero, Inequality \eqref{reto} gives $|H|\le 2$, which together with Inequality
     \eqref{ineqp} gives $d\le 5$, a contradiction.   Thus we can only have $\deg(Q_j)=0$, so that the unique other 
     $\omega_j$ that can exist  -- and actually exists because $\omega$ is not homogeneous -- is $\omega_1$. 
       Since $P_1$ is zero, there exists $\alpha\in\mathbb{C}^*$ such that  $$\omega=\alpha(xdy-ydx)+\beta dP.$$ Applying the substitution
   $(x,y)\leftarrow (\lambda x, \lambda y)$ with $\lambda^{10}=\alpha/\beta$ we obtain the 1-form
    $$\alpha\lambda^2\left(xdy-ydx+ dP\right),$$ so that $\F=\mathcal{P}_{11}$. \qed
    
   \section{Poincar\'e series and a Molien-type formula}\label{moliensec}
To study the foliations invariant by a primitive subgroup of $\pgl(3,\C)$ we change our approach to a more direct one through representations of these groups. We refer to \cite{invar} for the basic results of Invariant Theory.  

Let $M=\bigoplus_{d\geq0}M_d$ be a graded vector space such that $\dim_{\C}(M_d) <\infty$ for each $d$. The Poincar\'e series of $M$ is defined by
\[
P(M;t) = \sum_{d\geq0} \dim_{\C}(M_d)t^d.
\]
When $M$ is a finitely generated module over a complex polynomial ring, the Hilbert-Serre theorem says that $P(M,t)$ is a rational function of the variable $t$. Given a finite group $G$, a graded $G$-module $M$ and a character $\chi:G\longrightarrow \C^{\ast}$, the Reynolds operator $\mathcal{R}_{\chi}:M\longrightarrow M$ is defined by
\[
\mathcal{R}_{\chi}(m) = \frac{1}{|G|}\sum_{\varphi \in G}\chi(\varphi)^{-1}\varphi \cdot m , \, m\in M.
\]
It is an idempotent graded operator whose image is the submodule $M^{\chi}$ of $G$-semi-invariants with character $\chi$. In particular, the dimension of each homogeneous component of $M^{\chi}$ is given by the trace of the appropriate restriction of $\mathcal{R}_{\chi}$.  For instance, let $M= \C[X_0, \dots , X_n]$. For a faithful representation $G\longrightarrow \gl(n+1,\C)$ Molien's formula exhibits the rational function:
\begin{equation}\label{molien}
P(M^{\chi};t) = \frac{1}{|G|} \sum_{g\in G}\frac{\chi(g)}{\det(I-tg)}.
\end{equation}

Given a  finite subgroup $G<\gl(3,\mathbb{C})$ and a character $\chi:G\longrightarrow \C^{\ast}$, we want to describe the homogeneous vector fields $v$ on $\mathbb{C}^3$  with $\dv(v)=0$ that are  $G$-semi-invariant -- see Section \ref{definiciones}. Since the arguments are valid in any dimension, we state and prove our main result for a finite dimensional vector space $W$. Suppose that $G<\gl(W)$ is a finite group.   Thus, we are interested in the $G$-module $V$ of polynomial vector fields  whose divergence is zero. The action is given by the pushforward and the grading  is given by the degree. Recall that when $W= \C^3$ a degree d foliation on $\mathbb{P}^2$ is defined by an element of $V_d$ that is unique up to scalar multiplication. The following theorem is a Molien-type formula.
\begin{theorem}\label{molienthm}
	Let $W$ be a finite dimensional vector space, let $G < \gl(W)$ be a finite group and let $\chi:G\longrightarrow \C^{\ast}$ be a character. Then 
	\[
	P(V^{\chi};t) = \frac{1}{|G|} \sum_{\varphi\in G}\frac{\chi(\varphi)(\mbox{tr}(\varphi^{-1})-t)}{\det(I-t\varphi)}.
	\]
\end{theorem}

\begin{proof}
	First notice that the space of homogeneous polynomials of degree $d$ on $W$ is naturally identified with $S^{d}(W^\vee)$, the symmetric power of the dual space.
	The space of degree $d$ homogeneous vector fields is naturally identified with $S^{d}(W^\vee)\otimes W$. Let $v\in S^{d}(W^\vee)\otimes W$ be a homogeneous vector field. For any $\varphi \in \gl(W)$ we have $\dv(\varphi_{\ast}v) = \dv(v)\circ \varphi^{-1}$. Then $V_d$ is a $\gl(W)$-submodule of  $S^{d}(W^\vee)\otimes W$. On the other hand, for a homogeneous polynomial $P$ we have $\varphi_{\ast}(PR) = (P\circ \varphi^{-1})R$, where $R$ is the radial vector field. This gives us the exact sequence of $\gl(W)$-modules 
	\begin{equation}\label{glex2}
	0 \longrightarrow S^{d-1}(W^\vee) \longrightarrow S^{d}(W^\vee)\otimes W \longrightarrow V_d \longrightarrow 0.
	\end{equation}
	Note that this sequence is induced by the Euler sequence of $\mathbb{P}(W)$ hence, if $\dim W = n+1$,  $V_d \simeq {\rm H}^0(\mathbb{P}^n, T\mathbb{P}^n(d-1))$. The $\gl(W)$-action  in  (\ref{glex2}) is as follows. Given $p \in S^{d}(W^\vee)$, $w\in W$ and $\varphi \in \gl(W)$ we have  $\varphi \cdot w = \varphi(w)$, 
	$\varphi \cdot p = p \circ \varphi^{-1} = \mbox{S}^d(\varphi^{-1})^T(p) $ and
	 $\varphi \cdot (p\otimes w) = \varphi\cdot p\otimes \varphi\cdot w$, where
	$\mbox{S}^d(\varphi)$ is the matrix symmetric power and the $T$ superscript stands for the transpose. 
	
	Let us denote
	\[
	\C[W] = \bigoplus_{d\geq0} S^d(W^\vee).
	\]
	Then taking the direct sum on each term of the sequence \eqref{glex2} gives the sequence of graded $\gl(W)$-modules
	\[
	0 \longrightarrow \C[W] \longrightarrow \C[W]\otimes W \longrightarrow V \longrightarrow 0,
	\]
	where the first map has degree one and the second has degree zero. Applying the Reynolds operator associated to $G$ and $\chi$ we have that the sequence
	\[
	0 \longrightarrow \C[W]^\chi \longrightarrow (\C[W]\otimes W)^\chi \longrightarrow V^\chi \longrightarrow 0.
	\]
	is exact. From the degrees of the maps, the Poincaré series of $V^\chi$ is computed by
	\[
	P(V^{\chi};t) = P((\C[W]\otimes W)^\chi;t) - tP(\C[W]^\chi;t).
	\]
	We will now calculate the right hand side of this equation. From the definition of $\mathcal{R}_\chi$ we have	
	\[
	\dim(\C[W]_d\otimes W)^{\chi} = \frac{1}{|G|}\sum_{\varphi \in G} \tr(\chi(\varphi)^{-1}\mbox{S}^d(\varphi^{-1})^T\otimes\varphi) =  \frac{1}{|G|}\sum_{\varphi \in G} \chi(\varphi)^{-1}\tr(\varphi)\tr(\mbox{S}^d(\varphi^{-1})).
	\]
	On the other hand,
	\[
	\sum_{d\geq0}\chi(\varphi)^{-1}\tr(\varphi)\tr(\mbox{S}^d(\varphi^{-1}))t^d = \chi(\varphi)^{-1}\tr(\varphi)\sum_{d\geq0}\tr(\mbox{S}^d(\varphi^{-1}))t^d = \frac{\chi(\varphi)^{-1}\tr(\varphi)}{\det(I-t\varphi^{-1})}
	\]
	and it follows that
	\begin{align}\label{pve}
	P((\C[W]\otimes W)^\chi;t) & = \frac{1}{|G|}\sum_{\varphi\in G}\sum_{d\geq 0}\chi(\varphi)^{-1}\tr(\varphi)\tr(\mbox{S}^d(\varphi^{-1})t^d \nonumber\\
	& = \frac{1}{|G|}\sum_{\varphi\in G}\frac{\chi(\varphi)\tr(\varphi^{-1})}{\det(I-t\varphi)}.
	\end{align}
	The computation of $P(\C[W]^{\chi};t)$ follows analogously: 
	\begin{equation}\label{pker}
	P(\C[W]^{\chi};t) = \frac{1}{|G|} \sum_{\varphi\in G}\frac{\chi(\varphi)}{\det(I-t\varphi)}.
	\end{equation}	
	Combining \eqref{pve} and \eqref{pker} we conclude that 
	\[
	P(V^{\chi};t) = P((\C[W]\otimes W)^\chi;t) - tP(\C[W]^\chi;t) =\frac{1}{|G|} \sum_{\varphi\in G}\frac{\chi(\varphi)(\mbox{tr}(\varphi^{-1})-t)}{\det(I-t\varphi)}.
	\]	
\end{proof}

\section{Primitive groups}\label{primitive}
This section is devoted to prove propositions \ref{molien1} and \ref{molien2}. Let $\aut(\F)< \pgl(3,\C)$ be a transitive primitive group. We can take a finite group ${G}<\gl(3,\C)$ whose image in $\pgl(3,\C)$ is equal to 
$\aut(\F)$. Let $v$ be   a homogeneous vector field in $\mathbb{C}^3$ with $\dv (v)=0$ that induces $\F$. Then $v$ is $G$-semi-invariant. Note that this works for any choice of $G$ since $v$ is homogeneous. 

In the following subsections we analyze each of the possibilities -- described in the subsections \ref{tpg} and \ref{tpsg} -- for $\aut(\F)$.  For each of these cases we use  Theorem \ref{molienthm} to determine the
possible vector fields $v$ such that $|\aut(\F)|\ge 3(d^2+d+1)$. After our study, the proofs of the propositions \ref{molien1} and \ref{molien2} will be clearly finished.   All representations and character tables come from \cite{ATLAS}, except for the Hessian group that comes from \cite{BER}. We also remark that we have used Maple\texttrademark\, to speed up calculations of the Poincaré series, the same could be achieved in any computer algebra system or even by hand.

\subsection{The Hessian Group}To do the computations we will follow \cite{BER}.
Suppose that $\aut(\F)$ is the Hessian group. Then we can take $G$ equal to the triple cover of order $648$ in $\gl(3,\C)$ generated by the pseudo-reflections 
\[
R_1=\left(\begin{array}{ccc}
1 & 0 & 0 \\
0 & 1 & 0 \\
0 & 0 & \lambda^2
\end{array}\right), \, 
R_2=\frac{1}{\sqrt{-3}}\left(\begin{array}{ccc}
\lambda & \lambda^2 & \lambda^2 \\
\lambda^2 & \lambda & \lambda^2 \\
\lambda^2 & \lambda^2 & \lambda
\end{array}\right), \, 
R_3=\left(\begin{array}{ccc}
1 & 0 & 0 \\
0 & \lambda^2 & 0 \\
0 & 0 & 1
\end{array}\right),
\]
where $\lambda$ is a primitive cubic root of unity. The multiplicative characters are determined by their image of any element on the conjugacy class of $R_1$, see \cite{BER}, hence there are three possibilities: the trivial character $\chi_0(R_1)=1$ and $\chi_i(R_1)=\lambda^i$ for $i=1,2$. Then we compute the functions:

\begin{align}\label{hessfun}
P(V^G;t) & = \frac{t^{19}+t^{16}+t^{13}+t^{10}+t^7+t^4}{(1-t^9)(1-t^{12})(1-t^{18})}  = t^4+t^7+t^{10}+2t^{13}+3t^{16} + \dots \nonumber\\
P(V^{\chi_1};t) & = \frac{t^{31}+t^{28}+t^{25}+t^{22}+t^{19}+t^{16}}{(1-t^9)(1-t^{12})(1-t^{18})} =t^{16}+t^{19}+t^{22}+ 2t^{25}+3t^{28}+ \dots \nonumber\\
P(V^{\chi_2};t) & = \frac{-t^{37}+t^{28}+t^{25}+t^{22}+2t^{19}+t^{16}+t^{13}}{(1-t^9)(1-t^{12})(1-t^{18})} = t^{13}+t^{16}+2t^{19} + \dots
\end{align} 

The inequality  $216 = |\aut(\F)| \ge 3(d^2+d+1)$ implies that $d\leq7$. By the formulas (\ref{hessfun}),  the possible degrees are $d=4,7$ and there is exactly one foliation with each degree. Then, these foliations are necessarily the Hessian foliations  $\Hf_4$ and  $\mathcal{H}_7$.

\subsection{The normal subgroup $E$ of the Hessian Group}
Suppose that $\aut(\F)$ is the subgroup $E$ of the Hessian group as presented in Subsection \ref{tpg}. From the inequality $36 = |\aut(\F)| \ge 3(d^2+d+1)$, we see that $d=2$ is the only possibility. It is not difficult to prove that the orbits in $\mathbb{P}^2$  by the action of $E$ have at least $6$ elements.  Thus, since $\sing(\mathcal{F})$ is a union of orbits of $E$ and  $|\sing(\mathcal{F})|\le d^2+d+1=7$ we conclude that $\sing(\mathcal{F})$ is an orbit $\mathcal{O}$ of $E$. Then
$\F$ has at least $6$ singularities, all of them having the same Milnor number $\mu\in\mathbb{N}$. Since the sum of these Milnor numbers is equal to $7$, we necessarily have $\mu=1$ and $|\mathcal{O}|=7$, which is a contradiction because $7$ does not divide $36$.
Thus there is no foliation $\F$ satisfying  our hypotheses. 

\subsection{The normal subgroup $F$ of the Hessian Group}
Suppose that $\aut(\F)$ is the subgroup $F$ of the Hessian group as presented in Subsection \ref{tpg}. From the inequality $72 = |\aut(\F)| \ge 3(d^2+d+1)$, we obtain $d\le 4$. Clearly the set $\mathcal{O}_{12}\subset\mathbb{P}^2$  composed by the twelve singularities of the singular cubics of the Hesse pencil is invariant by $F$ -- see \cite{AD} for the explicit list of this points. By a straightforward computation we can verify that $\mathcal{O}_{12}$ is actually an orbit of $F$. Observe that the point $[0:0:1]\in \mathcal{O}_{12}$ is fixed by the subgroup $F_0\triangleleft F$ generated by $S$ and 
$$T^2V^2\colon [X:Y:Z]\mapsto[Y:X:Z].$$ It is easy to see that $F_0$ does not fix any line through $[0:0:1]$.  Then $[0:0:1]$ is necessarily a singularity of $\F$, otherwise the tangent line to $\F$ at $[0:0:1]$ would be fixed by $F_0$.  
Then the twelve points of $\mathcal{O}_{12}$ are singularities of $\F$. These singularities are simple because $\F$ has  at most $d^2+d+1\le 21$ singularities counting multiplicities. We can not have $\F=\mathcal{O}_{12}$ because the equation $d^2+d+1=12$ has no solutions.
Then $\F$ contains at least one orbit other than $\mathcal{O}_{12}$, which has at most nine points. We need the following fact about the group $F$, which is not difficult to prove:  the set of nine base points of the Hesse Pencil is an orbit of $F$ -- we denoted it by $\mathcal{O}_9$ --  and any other orbit has more than nine points.  Thus, we deduce that $\sing(\F)=\mathcal{O}_9\cup\mathcal{O}_{12}$ and that $\F$ has only simple singularities. Recall that a foliation with simple singularities is determined by its singular set -- see \cite{CO}. Then, since $\mathcal{H}_4$ has simple singularities and
$\sing(\mathcal{H}_4)=\mathcal{O}_9\cup\mathcal{O}_{12}$, we conclude that $\F=\mathcal{H}_4$.

\subsection{The Icosahedral group ${A}_5$} 
This group has a lift to $\SL(3,\C)$ that is isomorphic to itself. We can use the following presentation:
\[
G=\left< A,B | A^2=B^3=(AB)^5=1 \right>
\]
with generators
\[
A=\left(\begin{array}{ccc}
-1 & 0 & 0 \\
0 & -1 & 0 \\
r & r & 1
\end{array}\right), \, 
B=\left(\begin{array}{ccc}
0 & 1 & 0 \\
0 & 0 & 1 \\
1 & 0 & 0
\end{array}\right),
\]
where $r$ is the golden ratio $\frac{1+\sqrt{5}}{2}$. Since $G$ is a simple group, the only possible multiplicative character is the trivial one, hence

\begin{equation}\label{icosfun}
P(V^{G};t) = \frac{-t^{16}+t^{14}+t^{10}+t^9+t^6+t^5}{(1-t^{10})(1-t^6)(1-t^2)} = t^5 + t^6+ t^7+ t^8+ 2t^9+\dots 
\end{equation}
However, the inequality $60\ge 3(d^2+d+1)$ implies $d\leq 3$, so  that no foliation is under our hypotheses.

\subsection{The Klein Group $\mbox{PSL}(2,7)$} This group is the celebrated automorphism group of the Klein quartic of equation $X^3Y+Y^3Z+Z^3X=0$ which attains the Hurwitz bound; it has genus $3$ and $168=84(g-1)$ automorphisms. The Klein group is simple and it has a representation in $\SL(3,\C)$ given by
\[
G=\mbox{PSL}(2,7) =  \left< A,B | A^2=B^3=(AB)^7=[A,B]^4=1 \right>
\]
with generators
\[
A=\left(\begin{array}{ccc}
1 & -1-r & r \\
0 & -1 & 0 \\
0 & 0 & -1
\end{array}\right), \, 
B=\left(\begin{array}{ccc}
0 & 1 & 0 \\
0 & 0 & 1 \\
1 & 0 & 0
\end{array}\right),
\]
where $r=\frac{-1+\sqrt{-7}}{2}$. Since this group is simple,  we only have the trivial character and therefore

\begin{equation}\label{kleinfun}
P(V^{G};t) = \frac{-t^{22}+t^{18}+t^{16}+t^{11}+t^9+t^8}{(1-t^{14})(1-t^6)(1-t^4)} = t^8+ t^9+t^{11}+t^{12}+t^{13}+\dots
\end{equation}
The inequality $168\ge 3(d^2+d+1)$ implies that $d\leq 6$, so that no foliation satisfies our hypotheses.

\subsection{The Valentiner group ${A}_6$} This group has a perfect triple cover   in $\SL(3,\C)$ of order $1080$ given by
\[
G = \left< A,B | A^2=B^4=(AB)^{15}=(AB^2)^5=[(AB)^5,A]=1 \right>
\]
with generators
\[
A=\left(\begin{array}{ccc}
-1 & 0 & 0 \\
0 & 0 & 1 \\
0 & 1 & 0
\end{array}\right), \, 
B=\left(\begin{array}{ccc}
0 & 1 & 0 \\
-1 & 0 & 0 \\
\alpha & \beta & 1
\end{array}\right),
\]
where $\alpha=-\xi^7 -\xi^{13}$, $\beta=-\xi^2 -\xi^{8}$ for some  primitive $15$th root of unity $\xi$. Although the Valentiner group is not simple, it is perfect -- this means $[G,G]=G$. Then we have only the trivial character and therefore
\begin{equation}\label{valentfun}
P(V^{G};t) = \frac{-t^{46}+t^{40}+t^{34}+t^{25}+t^{19}+t^{16}}{(1-t^{30})(1-t^{12})(1-t^6)} = t^{16} +t^{19}+t^{22}+2t^{25}+ \dots
\end{equation}
Also in this case there is no foliation under our hypotheses, since the inequality $360\ge3(d^2+d+1)$ implies $d\leq 10$.

\end{document}